\documentclass[12pt]{amsart}
\usepackage{fullpage,url,amssymb,amsmath,graphicx}

\usepackage{amsmath, amsthm, amssymb}
\usepackage{url}
\usepackage{graphicx}
\usepackage {amsmath}
\usepackage{amsthm}
\usepackage {amssymb}
\usepackage{float}

\usepackage{amsmath}
\usepackage {amssymb}
\usepackage {graphicx,enumerate,booktabs}
\usepackage {color, tikz}
\usepackage[all]{xy}

\newtheorem{definition}{Definition}
\newtheorem{lemma}{Lemma}[section]
\newtheorem{theorem}[lemma]{Theorem}
\newtheorem*{theorem*}{Theorem}

\newtheorem{cor}[lemma]{Corollary}

\newtheorem{claim*}{Claim}
\newtheorem{remark}[lemma]{Remark}

\newtheorem{example}[lemma]{Example}
\newtheorem{conjecture}[lemma]{Conjecture}

\numberwithin{equation}{section}

\newcommand{\rank}{\operatorname{rank}}

\newcommand{\Osh}{{\mathcal O}}

\newcommand{\Sym}{\operatorname{Sym}}

\newcommand{\PP}{\mathbb{P}}

\newcommand{\RR}{\mathbb{R}}
\newcommand{\HH}{\mathcal{H}}

\newcommand{\W}{\mathcal{W}}

\def\H{\mathcal{H}}
\def\Sym{\operatorname{Sym}}
\def\a{\alpha}

\def\e{\eta}

\def\GL{\operatorname{GL}}

\begin{document}
\title{}
\date{}
\author{Jos\'e Acevedo}
\address{
Departamento de matem\'aticas\\
Universidad de los Andes\\
Carrera $1^{\rm ra}\#18A-12$\\ 
Bogot\'a, Colombia
}
\email{jg.acevedo36@uniandes.edu.co}

\author{
Mauricio Velasco\\ 
}
\address{
Departamento de matem\'aticas\\
Universidad de los Andes\\
Carrera $1^{\rm ra}\#18A-12$\\ 
Bogot\'a, Colombia
}
\email{mvelasco@uniandes.edu.co}
 % Reflection and Coxeter groups
\title{Witness sets for nonnegative polynomials }

\title{Test Sets for Nonnegativity of Polynomials Invariant under a Finite Reflection Group}
\begin{abstract} A set $S\subset \RR^n$ is a  nonnegativity witness for a set $U$ of real homogeneous polynomials if $F$ in $U$ is nonnegative on $\mathbb{R}^n$ if and only if it is nonnegative at all points of $S$. We prove that the union of the hyperplanes perpendicular to the elements of a root system $\Phi\subseteq \mathbb{R}^n$ is a witness set for nonnegativity of forms of low degree which are invariant under the reflection group defined by $\Phi$. We prove that our bound for the degree is sharp for all reflection groups which contain multiplication by $-1$. We then characterize subspaces of forms of arbitrarily high degree where this union of hyperplanes is a nonnegativity witness set. Finally we propose a conjectural generalization of Timofte's half-degree principle for finite reflection groups.
\end{abstract}

\maketitle

\section{ Introduction}
The general problem of deciding whether a homogeneous polynomial $F$ in $n$ variables with real coefficients is nonnegative has been one of the guiding questions of real algebraic geometry since the time of Hilbert. The intrinsic interest of this question has been complemented by a recent surge in its applications to polynomial optimization~\cite{BPT}, control theory~\cite{L} and to the general theory of aproximation algorithms~\cite{STEURER} among others. For all these applications, finding simple and efficient certificates for guaranteeing the nonnegativity of a given polynomial is of paramount importance.  

In the case of polynomials which are invariant under symmetric groups, such certificates have taken the form of ``nonnegativity witness sets", that is subsets $S\subset \RR^n$ such that nonnegativity of $F$ on $S$ is equivalent to the nonnegativity of $F$ in $\RR^n$. If the witness set $S$ is sufficiently simple then verifying that $F$ is nonnegative becomes a simpler problem. The work of several authors has provided important and interesting examples of nonnegativity witness sets for symmetric polynomials~\cite{Ch},~\cite{Ha},~\cite{Ti},~\cite{R}, ~\cite{IW}.  The following theorems are some representative results,
 
\begin{theorem*}[Choi-Lam-Reznick \cite{Ch}]
An $n$-ary even symmetric sextic is nonnegative if and only if it is so at the $n$ points $(1,0\ldots,0),(1,1,0,\ldots,0),\ldots,(1,1\ldots,1)$.
\end{theorem*}

\begin{theorem*}[Timofte's half-degree principle \cite{Ti},\cite{R}] An $n$-variate symmetric polynomial of degree $2d$ is nonnegative if and only if it is so at every point with at most $\max\{2,d\}$ distinct components. 
\end{theorem*}

The purpose of this note is to provide witness sets for polynomials which are invariant under the action of a finite reflection group. Our main finding is that the locus of points where a $\W$-invariant form achieves its minimum must contain the real points of certain families of $\W$-invariant curves. These curves,  constructed using the classical theory of finite reflection groups, must in turn intersect certain simple subsets which, as a result, can be used as witness sets for nonnegativity. To give a more precise description of the results of the article we need to establish some notation (see Section~$\S\ref{S: Prelims}$ for additional background and definitions).

Let $V$ be a real euclidean vector space of dimension $n$ and let $\Phi\subseteq V$ be a root system. Let $\W$ be the finite reflection group defined by $\Phi$. By a Theorem of Chevalley (see Theorem~\ref{Chevalley}) the algebra of polynomial functions on $V$ which are invariant under $\W$ (i.e. polynomials $f$ such that $f\circ g^{-1}=f$ for all $g\in \W$) is generated by $n$ basic homogeneous invariants of degrees $1\leq d_1\leq \dots \leq d_n$. Our first result is the following existence theorem for codimension one witness sets,

\begin{theorem} \label{Thm A} Assume $|x|^2$ is an element of the algebra generated by the first $n-1$ invariants. If $F$ is a $\W$-invariant form of degree $2d< 2d_n$ then $F$ is nonnegative on $V$ if and only if $F$ is nonnegative on the set $\HH_{\W}$, defined as the union of the hyperplanes that are perpendicular to the elements of $\Phi$.
\end{theorem}

If $\Phi:=\{e_i-e_j:1\leq i\neq j\leq n\}\subseteq \RR^n$ then the group $\W$ is the permutation group in $n$ letters and acts on the orthogonal complement of the subspace generated by $e_1+\dots+e_n$. This group with the action on that $n-1$ dimensional subspace is called $A_{n-1}$. Its basic invariants have degrees $2,\dots, n$ and Theorem~\ref{Thm A} says that a symmetric form $F$ of degree $2d<2n$ is nonnegative if and only if it is so on the hyperplanes $x_i-x_j=0$ (i.e. if and only if $F$ is nonnegative at all points with at most $n-1$ distinct components). If $\Phi=\{\pm e_i: 1\leq i \leq n\}\cup \{\pm e_i\pm e_j:1\leq i\neq j\leq n\}\subseteq \mathbb{R}^n$ then $\W$ is the Weyl group of $B_n$ acting on the polynomial ring in $n$ variables by signed permutations. The degrees of the basic invariants of $\W$ are $2,4,\dots, 2n$. In this case Theorem~\ref{Thm A} says that an even symmetric form of degree $2d<4n$ is nonnegative if and only if it is so on the hyperplanes defined by $x_i=0$ and by $\pm x_i\pm x_j=0$ for $1\leq i\neq j\leq n$. The statement of Theorem~\ref{Thm A} can be made more explicit by knowing the degrees of the basic invariants.  Table~\ref{Table1} contains these degrees for the reflection groups $\W$ generated by an irreducible root system (see~\cite[Chapter 2]{H} for a proof of the classification of finite reflection groups and for explicit descriptions of the roots in each case).

Next, it is natural to ask whether the degree bound in Theorem~\ref{Thm A} is sharp. We prove that this is the case when $\W$ is a reflection group all of whose basic invariants have even degree (i.e. for those reflection groups containing multiplication by $-1$). More generally, we prove the following existence theorem for $\W$-invariant forms for which the set $\H_\W$ is {\it not} a witness set for nonnegativity. 

\begin{theorem} \label{Thm B} Let $\underline{o}$ and $\overline{o}$ be the smallest and largest odd degrees of basic invariants for $\W$ and define $\underline{o}=\overline{o}=1$ if all invariants of $\W$ have even degree. If $2d\geq \max(2d_n, 2(\underline{o}+\overline{o}))$ then there exists a $\W$-invariant form $F$ of degree $2d$ which is nonnegative on $\H_\W$ and which is negative at some point of $V$.
\end{theorem}
In the case of irreducible reflection groups the bounds from the previous two Theorems can be computed explicity and we do so in Table~\ref{Table1}. Theorem~\ref{Thm B} implies that Theorem~\ref{Thm A} is sharp whenever the entries of the last two columns of Table~\ref{Table1} agree.
\begin{table}
\begin{center}
\caption{Degrees of basic invariants of irreducible reflection groups.}
\label{Table1}
\begin{tabular}{|c|l|l|l|}
\hline
Root System & Degrees $d_i$ & $2d_n$ & $\max(2d_n, 2(\underline{o}+\overline{o}))$\\
\hline
$A_n$ & $2,3,\dots, n+1$ & $2(n+1)$ &  $\begin{cases} 2(n+3)\text{, $n$ odd}\\ 2(n+4)\text{, $n$ even.}\\ \end{cases}$ \\
\hline
$B_n$ & $2,4,6,\dots, 2n$ & $4n$ & $4n$\\
\hline
$D_n$ & $2,4,6,\dots, 2n-2$,$n$ & $4n-4$ & $\begin{cases} 4n\text{, $n$ odd}\\ 4n-4\text{, $n$ even.}\end{cases}$\\ 
\hline
$E_6$ & $2,5,6,8,9,12$ & $24$ & $28$\\
\hline
$E_7$ & $2,6,8,10,12,14,18$ & $36$ & $36$\\
\hline
$E_8$ & $2,8,12,14,18,20,24,30$ & $60$ & $60$\\
\hline
$F_4$ &  $2,6,8,12$ & $24$  & $24$ \\
\hline
$G_2$ & $2,6$ & $12$ & $12$ \\
\hline
$H_3$ & $2,6,10$ & $20$ & $20$\\
\hline
$H_4$ & $2,12,20,30$ & $60$ & $60$ \\
\hline
$I_2(m)$ & $2,m$ & $2m$ & $\begin{cases} 4m\text{, $m$ odd}\\ 2m\text{, $m$ even.}\end{cases}$\\
\hline
\end{tabular}\\
\end{center}
\smallskip
\end{table}

By Theorem~\ref{Thm B} we know that the set $\H_\W$ above is not a witness set for {\it all} forms of sufficiently large degree. It is therefore natural to ask whether some subspaces of forms of higher degree have a natural nonnegativity witness set. Our next Theorem gives many such families of ``sparse" forms in the more general context of polynomials on forms, 

\begin{theorem}\label{CI} Suppose $g_1,\dots, g_n$ are a sequence of homogeneous forms in $\RR[x_1,\dots, x_n]$ and let $j$ be an integer such that $|x|^2\in \RR[g_1,\dots, g_{j}]$. If $F$ is a homogeneous polynomial of the form
\[F=A(g_1,\dots, g_j)+g_{j+1}B(g_1,\dots, g_j)\] 
for some polynomials $A,B \in \RR[z_1,\dots, z_j]$ then $F$ is nonnegative on $\mathbb{R}^n$ if and only if it is nonnegative on the algebraic set defined by the maximal minors of the matrix with columns $\nabla g_1,\dots ,\nabla g_{j+1}$.
\end{theorem}
\begin{cor}\label{corCI} Assume $g_1,\dots, g_n$ is a graded regular sequence in $\RR[x_1,\dots, x_n]$ of degrees $d_1\leq \dots\leq d_n$ and let $\mathcal{S}$ be the algebraic set defined by the determinant of the matrix with columns $\nabla g_1,\dots, \nabla g_n$. The set $\mathcal{S}$ is a witness set for nonnegativity of a form $F\in  \RR[g_1,\dots, g_n]$ if either
\begin{enumerate}
\item $|x|^2\in \RR[g_1,\dots, g_{n-1}]$ and $\deg(F)<2d_n$ or 
\item $F$ depends algebraically on no more than $n-1$ of the $g_i$ and the algebra generated by these $g_i$ contains $|x|^2$.  
\end{enumerate}
\end{cor}

By the Theorem of Chevalley, the basic invariants of a finite reflection group form a regular sequence and we can apply the Corollary above, obtaining a generalization of Theorem~\ref{Thm A}. For groups $\W$ which satisfy an additional ``minor factorization condition" (Definition~\ref{DefMF}) Theorem~\ref{CI} leads to higher codimension witness sets (see Corollary~\ref{highCodim}).

For the special case of polynomials invariant under the permutation group $S_n$, the previous Theorems are weaker than Timofte's remarkable half-degree principle~\cite{Ti},\cite{R}. It is natural to ask whether there is a generalization of the half-degree principle for all finite reflection groups. We conclude by proposing the following conjectural half-degree principle, which is a natural generalization of Theorem~\ref{Thm A}.

\begin{conjecture} \label{ABetter} Let $j\geq 1$ be an integer. Assume $|x|^2$ is in the algebra generated by the basic invariants $\eta_1,\dots, \eta_j$. If $F$ is a $\W$-invariant form of degree $2d$ with $d< d_{j}$ then $F$ is nonnegative on $V$ if and only if $F$ is nonnegative on the union of all subspaces obtained by intersecting $n-j+1$ linearly independent hyperplanes orthogonal to the elements of the root system $\Phi$ which defines $\W$.
\end{conjecture}

Proving Conjecture~\ref{ABetter} would imply both Timofte's half degree principle and the result of Choi, Lam and Reznick stated earlier on even symmetric sextics.

{\bf Acknowledgements} We wish to thank Grigoriy Blekherman, C\'esar Galindo, Mehdi Garrousian, Cordian Riener, Raman Sanyal and Gregory G. Smith for several insightful conversations during the completion of this project. M. Velasco was partially supported by the FAPA funds from Universidad de los Andes.

\section{Preliminaries}\label{S: Prelims}

\begin{definition} Let $V$ be a real vector space of dimension $n$ with a chosen positive definite bilinear form $\langle,\rangle$. The reflection across the hyperplane $H$ perpendicular to $v\in V$ is the element $T_H\in \GL(V)$ given by
\[ T_H(w)=w-2\frac{\langle w,v\rangle}{\langle v,v\rangle}\]
\end{definition}

\begin{definition} A root system $\Phi\subseteq V$ is a finite subset of vectors of $V$, called roots which satisfy the following properties,
\begin{enumerate}
\item{ The only multiples of $v\in \Phi$ which belong to $\Phi$ are $\{v,-v\}$.}
\item{ The inclusion $T_{H}(\Phi)\subseteq \Phi$ holds for any hyperplane $H$ perpendicular to an element of $\Phi$.}
\end{enumerate}
If moreover, $\forall x,y\in \Phi$  the real number $2\frac{\langle x,y\rangle}{\langle y,y\rangle}$ is an integer then the root system $\Phi$ is called crystallographic.
\end{definition}
\begin{definition}  The Reflection group $\W$ defined by a root system $\Phi$ is the subgroup of $\GL(V)$ generated by reflections across the hyperplanes perpendicular to the roots in $\Phi$. The rank of the group $\W$ is the dimension of the span of $\Phi$. A Weyl group is a reflection group defined by a crystallographic root system. \end{definition}

There is a complete classification of finite reflection groups (see~\cite[Chapter 2]{H}). Their most fundamental examples, pervasive throughout mathematics, are the Weyl groups. The classification of Weyl groups is as follows: there are three infinite one-parameter families $A_n, B_n, D_n$ for $n\in\mathbb{N}$ and five exceptional groups $E_6,E_7,E_8,F_4$ and $G_2$. The index $n$ always corresponds to the dimension of the vector space $V$ spanned by the corresponding root system. 

Let $S:=\Sym^{\bullet}(V^*)$ be the $\RR$-algebra of polynomial functions on $V$. Recall that $S$ is endowed with the contragradient action of $\W$, given by $w(f)(x):=f(w^{-1}x)$ and thus contains an invariant subalgebra $R:=S^{\W}\subseteq S$. The main theorem about invariants of finite reflection groups is the following Theorem (see~\cite[Chapter 3]{H} for a proof),
\begin{theorem}\label{Chevalley}[Chevalley] The algebra $R$ is generated by $n$ homogeneous, algebraically independent elements of positive degree. Moreover, $S$ is a free $R$-module of rank $|\W|$. 
\end{theorem}
Henceforth we will denote by $\eta_1,\dots, \eta_n$ a fixed set of homogeneous generators of $R$ of degrees $1\leq d_1\leq d_2\leq\dots \leq d_n$. It is well known (see~\cite[Section 3.9]{H}) that the degrees $d_i$ of these generators are uniquely determined by $\W$ and that they satisfy the equality 
\[\prod_{j=1}^n d_j=|\W|.\] 

A result of Shephard and Todd (see~\cite[Section 3.11]{H}) shows that the only finite subgroups of $\GL(V)$ whose ring of invariants is a polynomial algebra are generated by reflections so the Theorem of Chevalley is in fact a characterization of finite reflection groups.

\begin{definition} The Chevalley mapping of $\W$ is the morhism $\Psi: V\rightarrow \mathbb{A}^n$ which sends $x$ to $(\eta_1(x),\dots, \eta_n(x))$.
\end{definition}

We will often use the following factorization identity for the Jacobian determinant of the Chevalley mapping,

\begin{lemma}~\cite[Section 3.13]{H} \label{fact1}
For each root $\a\in\Phi$, define $l_{\a}$ to be a (nonzero) linear form that vanishes at the hyperplane orthogonal to $\alpha$. Then
\begin{align}
\nabla\e_1\wedge\ldots\wedge\nabla\e_n=\lambda \prod_{\a\in\Phi^+}{l_{\alpha}}
\end{align}
for some nonzero constant $\lambda\in \RR$. So $\nabla\e_1\wedge\ldots\wedge\nabla\e_n$ does not depend, up to multiplication by a nonzero constant, on the choice of basic invariants for $\W$. 
\end{lemma}

\section{Geometry of basic $\W$-invariant subvarieties}

\begin{definition} We say that a point $v\in V$ is $\W$-general if its $\W$-orbit has cardinality $|\W|$ (equivalently if its stabilizer subgroup is trivial). A point $v\in V$ is $\W$-special if it is not $\W$-general. The set of $\W$-special points is the union of all hyperplanes perpendicular to the roots. We denote this set by $\mathcal{H}_\W$. \end{definition}

By Lemma~\ref{fact1} the $\W$-general points can be characterized geometrically as the points where the Chevalley mapping $\Psi$ is a local dif and only ifeomorphism (i.e. the points where the Jacobian determinant of $\Psi$ does not vanish)

In this section we use $\W$-general points and the Chevalley mapping to construct certain families of $\W$-invariant subvarieties which we call basic. To study their geometry we first base-change to the algebraic closure. Let $\overline{V}=V\otimes_{\RR}\mathbb{C}$ be the complexification of $V$ and define $U=\mathbb{C}\times \overline{V}$ and let $\PP^n:=\PP(U)$ be the projective space over $\mathbb{C}$ with homogeneous coordinates $x_0,\dots, x_n$. The action of $\W$ on $V$ extends to an action of $\W$ on $\PP^n$ by fixing the $x_0$ coordinate.

\begin{definition} For $y\in V$ and $1\leq i\leq n$ let $\overline{\eta_i}=\eta_i(y)$. For any integer $k$ with $1\leq k\leq n$ let $\mathcal{Z}_y^k\subseteq U$ be the affine scheme defined by $(\eta_i-\overline{\eta_i}:1\leq i\leq k)$ and let $\hat{\mathcal{Z}}^k_y\subseteq \PP^n$ be the projective subscheme defined by $(\eta_i-\overline{\eta_i}x_0^{d_i}:1\leq i\leq k)$. 
\end{definition}

The following Lemma summarizes some basic geometric properties of the schemes $\mathcal{Z}_y^k$ and $\hat{\mathcal{Z}}_y^k$.

\begin{lemma} \label{Geom} The scheme $\hat{\mathcal{Z}}_y^k$ is $\W$-invariant, arithmetically Cohen-Macaulay of codimension $k$ and degree $\prod_{i=1}^kd_i$. In particular, $\hat{\mathcal{Z}}_y^k$ is unmixed and has no embedded components. 
Moreover, if $y$ is a $\W$-general point then $\hat{\mathcal{Z}}_y^k$ is reduced.
\end{lemma}
\begin{proof} We will show that the sequence $q_i:=\eta_i-\overline{\eta_i}x_0^{d_i}$, $i=1,\dots, n$ is a regular sequence. Since $\mathbb{C}[x_0,\dots, x_n]$ is an integral domain, $x_0$ is a nonzero divisor. Modulo $(x_0)$ the above sequence becomes $\eta_1,\dots, \eta_{n}$ in $\mathbb{C}[x_1,\dots,x_n]$ which is a regular sequence by~\cite[Corollary 5.3.4]{NS} and Theorem~\ref{Chevalley}. By~\cite[Corollary 17.2]{E} any permutation of a graded regular sequence is a regular sequence and thus $q_1,\dots, q_n$ is a regular sequence in $\mathbb{C}[x_0,\dots, x_n]$.  As a result, for any $k$, the scheme $\hat{\mathcal{Z}}_y^k$ has codimension $k$ and is arithmetically Cohen-Macaulay of degree equal to the product of the degrees of the $q_i$. By~\cite[Corollary 18.14]{E}, $\hat{\mathcal{Z}}_y^{k}$ is unmixed and has no embedded components. The fact that $\mathcal{Z}^k_y$ is $\W$-invariant is immediate from the fact that it is defined by $\W$-invariant forms. 
By Theorem~\ref{Chevalley}, the Chevalley mapping $\Psi: U\rightarrow \mathbb{A}^n$ is a finite and flat morphism of degree $|\W|$. Base-changing to the affine subspace $B\subseteq \mathbb{A}^n$ consisting of $(z_1,\dots, z_n)$ such that $z_i=\overline{\eta_i}$ for $1\leq i \leq k$ we obtain a finite and flat morphism $\Psi: \mathcal{Z}^{k}_y\rightarrow B$ of degree $|\W|$. As a result, every component of $\hat{\mathcal{Z}}_y^k$ surjects onto $B$ and it suffices to find a fiber of $\Psi$ consisting of a set of $|\W|$ smooth points to conclude that $\hat{\mathcal{Z}}_y^k$ is reduced. If $y$ is $|\W|$-general then, because the Jacobian of the original Chevalley mapping does not vanish at $y$, we conclude that $y$ is a nonsingular point of $\hat{\mathcal{Z}}_y^k$ and thus, that the $\W$-orbit of $y$ is the desired fiber. 
\end{proof}

The main idea of our proof will be to show that if $F$ is a homogeneous nonnegative form of sufficiently small degree and $y$ is a point where $F$ achieves its minimum on the unit sphere then the curve $\mathcal{Z}^{n-1}_y(\RR)$ is entirely contained on the locus of minima of $F$. As a result, to understand the minima of $F$ we can instead look at the geometry of the curves $\mathcal{Z}^{n-1}_y$. Our next Lemma gives a method to find a real point of these curves lying in the set $\H_\W$. 

\begin{lemma} \label{specialPt} Assume $|x|^2$ is an element of the algebra generated by the first $n-1$ invariants. For every $y\in V\setminus \{0\}$ the curve $\mathcal{C}_y:=\mathcal{Z}_y^{n-1}$ contains a point $x^*\in\H_\W\setminus \{0\}$.
\end{lemma}
\begin{proof} Since $|x|^2$ belongs to the algebra generated by the first $n-1$ invariants, the set of real points of the curve $\mathcal{C}_y$ is contained in the set of points $x$ with $|x|^2=|y|^2$. As a result $\mathcal{C}_y(\RR)$ is a nonempty compact set of nonzero points of $V$. Furthermore the optimization problem 
\[\min \eta_n(x): x\in \mathcal{C}_y(\RR)\] 
has a real solution $x^*$. Any such solution is a local minimum for $\eta_n(x)$ on $\mathcal{C}_y$ and thus, by the Lagrange multipliers Theorem there exists $\lambda\in \RR^{n}\setminus \{0\}$ such that
\[\lambda_1\nabla \eta_1(x^*)+\dots + \lambda_n\nabla\eta_n(x^*)=0\]  
and thus the Jacobian determinant of the Chevalley mapping vanishes at $x^*$. By Lemma~\ref{fact1} it follows that $x^*\in \H_\W$ as claimed. 
\end{proof}

\begin{remark}
The previous proof was inspired to us by ~\cite{Co}.
\end{remark}

\section{ Proof of Main Theorems}

In this section we prove our Main Theorems. The proof of the first Theorem is an application of Bezout's theorem on the reducible curve $\hat{\mathcal{Z}}_y^{n-1}$. It generalizes an argument due to H.W. Schülting~\cite{Ch} in the case of even symmetric sextics. The proof of the second Theorem is constructive and gives us a method to build forms of sufficiently high degree which are nonnegative on $\H_\W$ and not on $V$. 

\begin{proof}[Proof of Theorem \ref{Thm A}]
Suppose $F$ is nonnegative on $\H_\W$ and not on $V$. Let $y$ be a point of $V$ where $F$ reaches its minimum $\mu<0$ on the unit sphere $S$ of $V$. Note that $y$ must be a $\W$-general point. Let $G:=F-\mu|x|^{2d}$. The form $G$ is $\W$-invariant and nonnegative on $V$. We will show that if $2d<2d_n$ then there is a point $x^*\neq 0$ in $\H_\W$ which satisfies $G(x^*)=0$. This shows that $F(x^*)=\mu|x^*|^{2d}<0$ contradicting the fact that $F$ is nonnegative on $\H_\W$. Our proof relies on Bezout's theorem and thus we first base-change to the complex numbers and let $\mathcal{C}_y:=\hat{\mathcal{Z}}_y^{n-1}\subseteq \PP^n$ be the basic $\W$-invariant curve through $y$. By Lemma~\ref{Geom} we know that $\mathcal{C}_y$ is a reduced curve of degree $\prod_{i=1}^{n-1}d_i$. Since $G$ is $\W$ invariant we have two possibilities, either $G$ is identical to zero on $\mathcal{C}_y$ or it is nonzero on every irreducible component of $\mathcal{C}_y$. If $G$ is nonzero on every irreducible component of $\mathcal{C}_y$ then (see~\cite[\S 1.20]{K} for a detailed treatment of the degree theory of line bundles on a singular curve).
\[ 2d\deg(\mathcal{C}_y)=\deg_{\mathcal{C}_y}(\Osh_{\PP^n}(2d))=\sum_{p\in \mathcal{C}_y} e_p(G)\geq 2|\W|\]
where the last inequality occurs because every real zero of a nonnegative polynomial is a minimum, so every point in the orbit of $[1:y]$ contributes at least two to the degree of the line bundle. Since $|\W|=\prod_{i=1}^nd_i$ we conclude that $2d\geq 2d_n$ contradicting our assumption on the degree of the form $G$. We thus conclude that $G$ is identically zero in $\mathcal{C}_y$. By Lemma~\ref{specialPt}, the curve $\mathcal{C}_y$ contains a point $x^*\in \H_\W\setminus \{0\}$ proving our claim.
\end{proof}

\begin{example} The assumption that $|x|^2$ belongs to the algebra generated by the first $n-1$ invariants is necessary for Theorem~\ref{Thm A} to hold. If this hypothesis is not satisfied the statement may fail as the following example shows. Let $\Phi=\{\pm e_1\}\subseteq \RR^2$. The forms  $x_2,x_1^2$ are a set of basic invariants for $\W$. The form $-x_1^2$, of degree $2<4=2d_2$ is nonnegative on $\H_\W=\{x_1=0\}$ but not on V.   
\end{example}

The previous proof implies the following Theorem of complex algebraic geometry, of interest in its own right. It is interesting to ask whether similar statements hold for $\W$-invariant varieties of higher codimension.

\begin{theorem}  If $F$ is a $\W$-invariant hypersurface of degree $d<2d_n$ which is singular at a point $y$ where the Jacobian determinant of the Chevalley mapping does not vanish then $V(F)$ must contain the curve $\mathcal{C}_y$. 
\end{theorem}

\begin{proof}[Proof of Theorem \ref{Thm B}] Let $y\in V$ be a $\W$-general point on the unit sphere $S$ of $V$ where none of the basic invariants of $\W$ vanishes. Let $X:={\rm Cone}(\W y)\subseteq V$ be the cone over the $\W$-orbit of $y$. Let $J$ be any homogeneous ideal with $V_\RR(J)=X$ and let $\beta$ be the maximum degree of one of its generators. We will construct a homogeneous $\W$-invariant form of degree $2\beta$ which is nonnegative on $\H_\W$ and which is strictly negative on $y$. To this end, let $h_1,\dots, h_s$ be a generating set of the vector space $J_\beta$ which is closed under the action of $\W$. Define $\phi:=h_1^2+\dots + h_s^2$ and note that $\phi$ is $\W$-invariant, nonnegative, and satisfies $V_\RR(\phi)=X$. Let $\mu=\min\{\phi(x): x\in \H_\W\cap S\}$. Then $\mu>0$ and thus $\overline{\phi}:=h_1^2+\dots+h_s^2-\frac{\mu}{2}|x|^{2\beta}$ is a homogeneous, $\W$-invariant form, which is nonnegative on $\H_\W$ and satisfies $\overline{\phi}(y)=-\frac{\mu}{2}<0$. Next, we obtain an upper bound for the number $\beta$ by constructing homogeneous ideals $J$ with $V_\RR(J)=X$. 

Let $O=\{j: 1\leq j\leq n\text{ and }\deg(\eta_j)\text{ is odd} \}$. If $O\neq \emptyset$ then let $\underline{o}:=\min O$ and $\overline{o}=\max O$ and for a basic invariant $\eta_j$ define $\overline{\eta_j}:=\eta_j(y)$. Define homogeneous polynomials $p_j$  for $1\leq j\leq n$ by
\[
p_j:=\begin{cases}
\eta_j-\overline{\eta_j}|x|^{d_j}\text{, if $j\not\in O$}\\
\eta_j\eta_{\underline o}-\overline{\eta_j} \overline{\eta_{\underline o}}|x|^{d_j+\underline{o}}\text{, if $j\in O$}\\
\end{cases}
\] 
We claim that $V_\RR(p_1,\dots, p_n)=X$. Equivalently we will show that $V_\RR(p_1,\dots, p_n)\cap S$ equals $\W y\cup -\W y$. Since $V_\RR(p_1,\dots, p_n)$ is generated by homogeneous $\W$-invariant forms of even degree it is immediate that $\W y \cup -\W y\subseteq V_\RR(p_1,\dots, p_n)\cap S$. For the opposite inclusion, if $x\in V_\RR(p_1,\dots, p_n)\cap S$ then $\eta_{\underline{o}}(x)^2=\eta_{\underline{o}}(y)^2$ and thus $\eta_{\underline{o}}(x)=\pm \eta_{\underline{o}}(y)$. It follows that either $\eta_i(x)=\eta_i(y)$ for all $i$ with $1\leq i\leq n$ or $\eta_i(x)=\eta_i(-y)$ for all $i$ with $1\leq i\leq n$ so that $x\in \W y\cup -\W y$ (note that the two alternatives coincide if $\W$ has no invariants of odd degree). The largest degree of a polynomial $p_j$ is $\max(d_n, \underline{o}+\overline{o})$ and this is an upper bound for $\beta$ in the previous paragraph. 
As a result, the $\W$-invariant form $\overline{\phi}$ obtained from the ideal $J=(p_1,\dots, p_n)$ has degree $2\max(d_n,\underline{o}+\overline{o})$, is nonnegative on $\H_\W$ and is strictly negative on $y$ as claimed.
\end{proof}

\begin{proof}[Proof of Theorem \ref{CI}] Suppose that $F$ is nonnegative on the algebraic set defined by the maximal minors of the matrix with columns $\nabla g_1,\dots, \nabla g_j$ and not on $V$. Let $\mu<0$ be the absolute minimum of $F$ on the unit sphere in $V$ and let $y$ be a point where this minimum is achieved. Define $G:=F-\mu|x|^{2d}$ where $2d=\deg(F)$ and note that $G$ is nonnegative on $V$. Let $\overline{g_i}=g(y_i)$ and let $\mathcal{Z}$ be the variety defined by $(g_i-\overline{g_i}: 1\leq i\leq j)$. Since $|x|^2$ is an element of $\RR[g_1,\dots, g_j]$ we know that $\mathcal{Z}(\RR)$ is a closed subset of the unit sphere in $V$ and thus compact. Moreover, letting $\overline{A}:= A(\overline{g_1},\dots, \overline{g_j})$ and $\overline{B}:=B(\overline{g_1},\dots, \overline{g_j})$ the form $G$ restricts to a function $\phi$ on $\mathcal{Z}$ given by
\[ \phi(x)=\overline{A}-\mu|y|^{2d}+g_{j+1}(x)\overline{B}\]
which satisfies $\phi(y)=0$ and $\phi(x)\geq 0$ for all $x\in \mathcal{Z}$. We have two cases, either $\overline{B}=0$ and $\phi$ is identical to zero in $\mathcal{Z}$ or $\overline{B}\neq 0$. We will show that we reach a contradiction in both cases.
In the first case, arguing as in Lemma~\ref{specialPt} we conclude that there is a point $x^*\in \mathcal{Z}(\RR)$ where the maximal minors of $\nabla g_1,\dots, \nabla g_j$ vanish. If $\overline{B}\neq 0$ then the point $y$ is a global and hence local minimum of $\phi$ on $\mathcal{Z}$. By the Theorem of Lagrange multipliers there exist $(\lambda_0,\dots, \lambda_j)\neq 0$ such that $\lambda_0\nabla \phi(y) +\sum_{i=1}^j\lambda_i\nabla g_i(y)=0$. Since $\nabla\phi = \overline{B}\nabla g_{j+1}$ it follows that the matrix with columns $\nabla g_1(y),\dots, \nabla g_{j+1}(y)$ has a nonzero element in its kernel and thus that $y$ belongs to the algebraic set defined by the maximal minors of the matrix with columns $\nabla g_1,\dots \nabla g_j$, a contradiction.
\end{proof}

\begin{proof}[Proof of Corollary~\ref{corCI}] $(1)$ Since the forms $g_1,\dots, g_n$ are a regular sequence the algebra $\RR[g_1,\dots, g_n]$ is isomorphic to a graded polynomial ring. As a result, for every homogeneous element $F\in \RR[g_1,\dots, g_n]$ there exists a unique polynomial $H\in \RR[z_1,\dots, z_n]$ such that $F=H(g_1,\dots, g_n)$. If $\deg(F)<2d_n$ then no monomial in $H$ can be divisible by $z_n^2$ and $H=A(z_1,\dots, z_{n-1})+z_n B(z_1,\dots, z_{n-1})$. The claim follows from Theorem~\ref{CI}.  $(2)$ Follows immediately from Theorem~\ref{CI}.
\end{proof}

Next we analyze the consequences of Theorem~\ref{CI} on forms which are invariant under a reflection group $\W$.

\begin{definition}\label{DefMF} A sequence of basic invariants $\eta_1,\dots \eta_n$ satisfy the minor factorization condition if for every $x\in V$ and every $j\in 1,\dots, n$ the following equality holds
\[\rank\left(\nabla \eta_1(x),\dots,\nabla \eta_{n}(x)\right)\leq j-1 \iff \rank \left(\nabla \eta_1(x),\dots,\nabla \eta_{j}(x)\right)\leq j-1\]
\end{definition}
The motivation for the terminology comes from the fact that, in some cases, the equivalence is a consequence of a factorization property of the minors of the Jacobian of the Chevalley mapping. The following example shows that this is the case for the Weyl group of $B_n$. An analogous argument shows that this property also holds for type $A$.

\begin{example} Let $\Phi=\{\pm e_i: 1\leq i \leq n\}\cup \{\pm e_i\pm e_j:1\leq i\neq j\leq n\}\subseteq \mathbb{R}^n$ and let $\eta_i:=\sum_{j=1}^n x_j^{2i}$. The group $\W$ is the Weyl group of $B_n$ and $\eta_1,\dots, \eta_n$ are a set of basic invariants. For $1\leq i_1< \dots < i_k\leq n$ and $1\leq j_1< \dots < j_k\leq n$ Let $\Delta_{i_1,\dots, i_k}^{j_1,\dots, j_k}$ be the determinant of the submatrix of the Jacobian of the Chevalley mapping with rows indexed by $i_1,\dots, i_k$ and columns indexed $j_1,\dots, j_k$. We show that the minor factorization property holds by proving that for every $i_1,\dots,i_k$ and $j_1,\dots, j_k$, the polynomial  $\Delta_{i_1,\dots,i_k}^{1,\dots, k}$ divides $\Delta_{i_1,\dots,i_k}^{j_1,\dots, j_k}$. To see this, note that the subgroup of $B_n$ which fixes the components $x_{s_j}$ for $s_j\in[n]\setminus \{i_1,\dots, i_k\}$ is isomorphic to $B_k$ and has as set of fundamental invariants $\eta_1,\dots, \eta_k, x_{s_1},\dots, x_{s_{n-k}}$. Thus, the Jacobian determinant of its Chevalley mapping is precisely $J=\Delta_{i_1,\dots, i_k}^{1,\dots, k}$. On the other hand, the determinant of the submatrix  of $\nabla \eta_{j_1},\dots, \nabla \eta_{j_k}, \nabla x_{s_1},\dots, \nabla x_{s_{n-k}}$ with rows $i_1,\dots, i_k$ is an alternating function for $B_k$ and thus, by Proposition~\cite[pag. 69]{H} its determinant $\Delta_{i_1,\dots, i_k}^{j_1,\dots, j_k}$ is divisible by $J$ as claimed.
\end{example}

\begin{remark} There are several possible explanations for the above minor factorization conditions. In the $A_{n-1}$ and $BC_n$ cases the ratios between the various minors are Schur polynomials. But this is not always the case, for example, taking a explicit set of basic invariants for $D_3$, for instance $x_1^2+x_2^2+x_3^2$, $x_1x_2x_3$ and $x_1^4+x_2^4+x_3^4$~\cite{Iwa}, the ratios between the various minors is not a polynomial. However, the vanishing locus of the maximal minors corresponding to the first $k$ columns of their Jacobian is the same as the vanishing locus of all $k\times k$ minors of the Jacobian for each $k=1,2,3$ (see the next Example). We conjecture that both vanishing loci are the same for each $k=1,2,\dots,n$ for a finite reflection group of rank $n$, i.e., the minor factorization property holds for every finite reflection group. \end{remark}

\begin{example}\label{locus}
Consider the jacobian matrix $J_{D_3}$ of the Chevalley map of $D_3$ corresponding to the basic invariants above.
$$J_{D_3}(x)=\arraycolsep=8pt\def\arraystretch{2.2}\left[\begin{array}{c c c}
2x_1 & x_2x_3 & 4x_1^3\\
2x_2 & x_1x_3 & 4x_2^3\\
2x_3 & x_1x_2 & 4x_3^3
\end{array}\right]$$
\end{example}
The vanishing locus of the $2\times2$ minors of the first two columns is given by the solution set to the following system:
\begin{eqnarray*}
x_3(x_1^2-x_2^2)=0\\
x_2(x_1^2-x_3^2)=0\\
x_1(x_2^2-x_3^2)=0
\end{eqnarray*}
If one of the variables is zero then another one must be zero as well and therefore the third one can assume any value. If none of the $x_i$ vanish then all must have the same square. Therefore the vanishing locus is $$\{(t,0,0),(0,t,0),(0,0,t),(t,t,t),(t,t,-t),(t,-t,t),(t,-t,-t):t\in\RR\}.$$ Up to symmetry, for a given $t$, there are three kinds of points in the locus above: $(t,0,0)$, $(t,t,t)$ and $(t,t,-t)$. Plugging in each one of them in the matrix above we get:
$$\arraycolsep=8pt\def\arraystretch{2.2}\left[\begin{array}{c c c}
2t & 0 & 4t^3\\
0 & 0 & 0\\
0 & 0 & 0
\end{array}\right]\,\,\,\,\,\,\,\,\,\,\,\arraycolsep=8pt\def\arraystretch{2.2}\left[\begin{array}{c c c}
2t & t^2 & 4t^3\\
2t & t^2 & 4t^3\\
2t & t^2 & 4t^3
\end{array}\right]\,\,\,\,\,\,\,\,\,\,\,\arraycolsep=8pt\def\arraystretch{2.2}\left[\begin{array}{c c c}
2t & -t^2 & 4t^3\\
2t & -t^2 & 4t^3\\
-2t & t^2 & -4t^3
\end{array}\right]$$

Since all $2\times2$ minors of each matrix vanish, we conclude that the vanishing locus of all $2\times2$ minors of $J_{D_3}(x)$ coincides with the vanishing locus of the $2\times2$ minors of the first two columns.

\begin{cor} \label{highCodim}Assume $\eta_1,\dots \eta_n$ satisfy the minor factorization property and that $|x|^2$ is in the algebra generated by the basic invariants $\eta_1,\dots, \eta_{j-1}$. Define $\mathcal{S}$ as the union of all subspaces obtained by intersecting $n-j+1$ linearly independent hyperplanes orthogonal to the elements of the root system $\Phi$ which defines $\W$ (i.e., the set of $(j-1)$-dimensional faces of $\H_\W$). The set $\mathcal{S}$ is a witness set for nonnegativity of any form $F\in  \RR[\eta_1,\dots, \eta_j]$ which is linear in $\eta_j$.
\end{cor}
\begin{proof} Arguing as in the proof of Corollary~\ref{corCI} the claim reduces to proving that the locus of points $x$ where $\rank(\nabla \eta_1(x),\dots, \nabla \eta_j(x))<j$ coincides with $\mathcal{S}$. By the minor factorization property this set consists of points where the rank of the Jacobian of the Chevalley mapping is at most $j-1$ which coincides with $\mathcal{S}$ by a well known Theorem of Steinberg~\cite{St}. \end{proof}

Since $B_n$ satisfies the minor factorization property, part $(a)$ of the previous Corollary implies that an even symmetric sextic is nonnegative if and only if it is so on the intersection of any $n-2$ linearly independent hyperplanes perpendicular to elements of $B_n$ and that an even symmetric quartic is nonnegative if and only if it is so at the $n$ points $(1,0\ldots,0),(1,1,0,\ldots,0),\ldots,(1,1\ldots,1)$ which, up to the action of $B_n$, span the subspaces that can be obtained by intersecting any $n-1$ linearly independent hyperplanes perpendicular to elements of $B_n$.

\begin{remark} As this article was being completed we learned of independent in progress work by Friedl, Riener and Sanyal~\cite{RSF} where they provide the first high codimension witness sets valid for all finite reflection groups. 
\end{remark}

\end{document}